\RequirePackage[l2tabu, orthodox]{nag}
\documentclass[preprint,  11pt]{amsart}

\usepackage{fullpage}
\usepackage{graphicx}

\usepackage[square,numbers]{natbib}

\usepackage{amsfonts}
\usepackage{amssymb}
\usepackage{amsmath}
\usepackage{amsthm}
\usepackage{amsbsy}
\usepackage{amssymb} 
\usepackage{verbatim}
\usepackage{bm} 
\usepackage{paralist} 
 \usepackage{color}
 \usepackage{mathrsfs}
 \usepackage{graphicx}
\usepackage{hyperref}

\def\<{\langle}
\def\>{\rangle}

\newcommand{\ba}{\[\begin{aligned}}
\newcommand{\ea}{\end{aligned}\]}
\newcommand\mnote[1]{} 
\newcommand{\beq}[1]{\begin{equation}\label{#1}}
\newcommand\eeq{\end{equation}}
\newcommand\ben{\begin{equation}}
\newcommand\een{\end{equation}}
\newcommand\bes{\begin{eqnarray*}}
\newcommand\ees{\end{eqnarray*}}
\newcommand\besn{\begin{eqnarray}}
\newcommand\eesn{\end{eqnarray}}

\def\bthm{\begin{theorem}}
\def\ethm{\end{theorem}}
\def\bdefn{\begin{definition}}
\def\edefn{\end{definition}}
\newcommand{\benu}{\begin{enumerate}\setlength\itemsep{6pt}}
\newcommand{\beit}{\begin{itemize}\setlength\itemsep{3pt}}
\def\eenu{\end{enumerate}}
\def\eeit{\end{itemize}}
\def\beds{\begin{description}}
\def\eeds{\end{description}}
\def\bepr{\begin{problem}}
\def\eepr{\end{problem}}

\def\bprf{\begin{proof}}
\def\eprf{\end{proof}}
\def\berk{\begin{remark}}
\def\eerk{\end{remark}}
\def\bex{\begin{exercise}}
\def\eex{\end{exercise}}
\def\beg{\begin{example}}
\def\eeg{\end{example}}

\def\suchthat{{\; : \;}}

\def\PP{{\mathcal P}}

\def\N{\mathbb{N}}

\newcommand{\sm}{{\raise0.3ex\hbox{$\scriptstyle \setminus$}}}

\def\alp{\alpha}

\renewcommand\phi{\varphi}

\theoremstyle{plain} 
    \newtheorem{theorem}{Theorem}

\theoremstyle{definition} 
    \newtheorem{definition}[theorem]{Definition}

    \newtheorem{exercise}[theorem]{Exercise}
    \newtheorem{problem}[theorem]{Problem}
        \newtheorem{remark}[theorem]{Remark}
    \newtheorem{example}[theorem]{Example}

\renewcommand{\bex}{\indent\begin{exercise}}
\renewcommand\subset{\subseteq}

\usepackage{palatino}

\hypersetup{
    colorlinks=true, 
    linktoc=all,     
    linkcolor=blue,  
}
\begin{document}

\title{On Hadamard powers of positive semi-definite matrices}
\author[J.~S.~Baslingker]{Jnaneshwar Baslingker}
\author[B. Dan]{Biltu Dan}
\address{Department of Mathematics, Indian Institute of Science, Bangalore-560012, India.}
\email{jnaneshwarb@iisc.ac.in, biltudan@iisc.ac.in}
\date{}
\begin{abstract}
Consider the set of scalars $\alpha$ for which the $\alpha$th Hadamard power of any $n\times n$ positive semi-definite (p.s.d.) matrix with non-negative entries is p.s.d. It is known that this set is of the form $\{0, 1, \dots, n-3\}\cup [n-2, \infty)$. A natural question is ``what is the possible form of the set of such $\alpha$ for a fixed p.s.d. matrix with non-negative entries?". In all examples appearing in the literature, the set turns out to be union of a finite set and a semi-infinite interval. In this article, examples of matrices are given for which the set consists of a finite set and more than one disjoint interval of positive length. In fact, it is proved that the number of such disjoint intervals can be made arbitrarily large, by giving explicit examples of matrices.
The case when the entries of the matrices are not necessarily non-negative is also considered. 
\end{abstract}
\keywords{ Positive semi-definite, Hadamard power}
\subjclass[2010]{15B48, 15A45 }
\maketitle
\section{Introduction}

Entrywise functions of matrices preserving positive semi-definiteness has been a topic of active research. The long history of this field starts with the Schur product theorem (see \cite{HJ}), followed by the works of Schoenberg\cite{IS}, Rudin\cite{WR} and others. In particular, the study of entrywise power functions $x\rightarrow x^\alpha$ has been of special interest to several mathematicians (see~\cite{BE,FH77,GKB,hiai, jain}). By the Schur product theorem, the $m$th Hadamard power $A^{\circ m}:=[a_{ij}^m]$ of any p.s.d. matrix $A=[a_{ij}]$ is again p.s.d. for every positive integer $m$. Let $\PP_n^+$ denote the set of $n\times n$ p.s.d. matrices with non-negative entries. For $A=[a_{ij}]\in \PP_n^+$ and $\alpha \geq 0$, the $\alpha$th Hadamard power of $A$ is the matrix $A^{\circ \alpha}:=[a_{ij}^\alpha]$. We use the convention $0^0=1$. Now for $A\in \PP_n^+$, we define $$S_A:=\{\alp\ge 0\suchthat A^{\circ \alpha}\in \PP_n^+\}.$$ By the Schur product theorem we have that all the natural numbers belong to $S_A$ for every $A\in \PP_n^+$. Also, $0$ belongs to $S_A$ for every $A\in \PP_n^+$, as the matrix with all entries $1$ is p.s.d.  FitzGerald and Horn \cite{FH77} showed that $n-2$ is the `critical exponent' for such matrices, i.e., $n-2$ is the least number for which $A^{\circ \alpha}\in \PP_n^+$ for every $A\in \PP_n^+$ and for every $\alpha \ge n-2$. Thus $\bigcap_{A\in \PP_n^+}S_A=\{0,1,\dots, n-3\}\cup[n-2,\infty)$. They considered the matrix $A\in \PP_n^+$ with $(i,j)$th entry $1+\epsilon ij$ and showed that if $\alpha$ is not an integer and $0<\alpha<n-2$, then $A^{\circ \alpha}$ is not positive semi-definite for a sufficiently small positive number $\epsilon$. Later, Jain~\cite{jain} showed that this remains true if $\epsilon ij$ is replaced with $x_ix_j$ for any distinct positive real numbers $x_1, \dots , x_n$.

What is the possible form of the set $S_A$ for a fixed matrix $A\in\PP_n^+$? In all examples appearing in the literature, the set turns out to be the union of a finite set and a semi-infinite interval. Such sets have appeared in different areas like probability, complex analysis and representation theory, under the names Wallach set and Berezin-Wallach set. Gindikin set, which also has a similar structure, shows up in random matrix theory. For more details see Section $1.2$ of \cite{khare2020}. 

We first note down the following properties of $S_A$ for any $A\in\PP_n^+$ for any $n$.
\begin{enumerate}
\item $\N \cup \{0\} \subset S_A$.
\item If $\alpha, \beta\in S_A$, then $\alpha+\beta \in S_A$. Thus, $S_A$ is a monoid under addition.
\item If $[0, \delta] \subset S_A$ for some $\delta>0$, then the monoid property ensures that $S_A= [0, \infty)$.
\item $S_A$ has finitely many connected components (see Remark~\ref{rem:3}). One of the components is a semi-infinite interval containing $[n-2, \infty)$.
\item $S_A$ is a closed subset of $[0, \infty)$.
\end{enumerate} 
The first two properties follow from the Schur product theorem and the last one follows from the continuity of the smallest eigenvalue. 

In all examples appearing in the literature, for fixed $A\in\PP_n^+$, the set $S_A$ has only one interval (semi-infinite) component. The rest of the components are all singleton sets. We show that the structure of $S_A$ need not always be a semi-infinite interval together with a finite set. One can observe that if $A\in \PP_2^+$, then $S_A=[0,\,\infty)$. If $A\in \PP_3^+$, then $S_A=\{0\} \cup [\rho,\,\infty)$ for some $\rho\in[0,\,1]$. Indeed, it follows from Theorem 2.2 of~\cite{FH77} that $[1,\,\infty) \subset S_A$ and if $A^{\circ \alpha}$ is p.s.d. for $\alpha>0$, then $A^{\circ \beta}=(A^{\circ \alpha})^{\circ \beta/\alpha}$ is p.s.d. for any $\beta>\alpha$. For $n\geq 4$, we give examples of $A\in \PP_n^+$ such that $S_A$ has at least two interval\footnote{ By interval we always mean interval of positive length.} components.

We now consider the matrix \begin{align}\label{def:A_n} A_n:= [1+ x_ix_j]_{1\leq i,j\leq n}\end{align} for distinct positive real numbers $x_1,  x_2, \dots, x_n$. Theorem 1.1 of~\cite{jain} showed that $S_{A_n}= \{0, 1, \dots, n-3\}\cup [n-2, \infty)$. Let  $n\ge 4$. For $k\in\{2,3,\dots, n-2\}$ define
\begin{align*}
A_{n,k,\epsilon}:=A_n+ \epsilon R_{n,k}, \,\,\,\epsilon \ge 0,
\end{align*}
where $R_{n,k}$ is the $n\times n$ matrix with $1$ in the last $k$ diagonal entries and $0$ otherwise. Note that $A_{n,k,\epsilon}$ is  p.s.d. We prove that for small enough $\epsilon$, $S_{A_{n,k,\epsilon}}$ has at least $k$ interval components, one around each of the integers from $n-k-1$ to $n-3$ and the last one is the semi-infinite interval containing $n-2$.
\begin{theorem}\label{thm1}
Fix integers $n\geq 4$ and $2\le k\le n-2$, and let the matrix $A_n$ be as in~\eqref{def:A_n}. Let $\epsilon>0$ be sufficiently small (depending on $A_n$). Then there exists $\delta>0$ such that $[\ell-\delta, \ell+\delta]\subset S_{A_{n,k,\epsilon}}$ for any $\ell\in\{n-k-1, n-k,\dots,n-3\}$. Also, for each such $\ell$ there exists $\alpha_\ell\in (\ell, \ell+1)$ such that $\alpha_\ell\notin S_{A_{n,k,\epsilon}}$. Moreover, $S_{A_{n,k,\epsilon}} \cap [0, n-k-2] =\{0,1,\dots, n-k-2\}$.
\end{theorem}

\begin{remark}
Note that $S_{A_{n, k, \epsilon}}$ may have more than $k$ interval components  but all such extra intervals should be contained in $[n-k-2,\,n-2]$. In other words, there is no interval before $n-k-2$.
\end{remark}
The proof of the above theorem is given in Section~\ref{sec:proofs}. The proof works for $k=n-1$ and $n$ as well. In those cases, there is interval around each of the integers from $1$ to $n-2$. Note that an interval containing $0$ is not possible in any of the above examples. We make the following remark on the maximum number of possible intervals in $S_A$ for  $A\in\PP_n^+$.
\begin{remark}\label{rem:3}
For any $A\in \PP_n^+$, $S_A$ can have at most $n!$ interval components. Indeed, note that the determinant of any $m\times m$ principal submatrix of $A^{\circ \alpha}$ is an exponential polynomial in $\alpha$ with at most $m!$ sign changes. Therefore by Proposition 3.2 of ~\cite{jain}, the number of zeros of determinants (as functions of $\alpha$) of all principal submatrices  combined is at most $\sum_{m=1}^n m!$, which is bounded by $2n!$. This implies that the total number of intervals in $S_A$ can be at most $n!$.
\end{remark}

If $A$ has arbitrary real (not necessarily non-negative) entries, then a natural extension of real Hadamard powers is now considered. Let $\PP_n$ denote the set of $n\times n$ p.s.d. matrices with arbitrary real entries. For $A=[a_{ij}]\in \PP_n$ and $\alpha \geq 0$, we define the matrix $|A|_{\circ}^{\circ \alpha}:=[|a_{ij}|^\alpha]$. In particular if $\alpha = 1$, then we denote the matrix $|A|_{\circ}^{\circ \alpha}$ by $|A|_{\circ}$. Now for $A\in \PP_n$, we define $$S_{|A|_{\circ}}:=\{\alp\ge 0\suchthat |A|_{\circ}^{\circ \alpha}\in \PP_n\}.$$ By the Schur product theorem and the fact that the matrix with all entries $1$ is p.s.d., we have that all the non-negative even integers belong to $S_{|A|_{\circ}}$ for every $A\in \PP_n$. Hiai (\cite{hiai}) proved an analogue of the
theorem of FitzGerald and Horn for $n \times n$ real positive semi-definite
matrices. He showed that $n-2$ is the least number for which $|A|_{\circ}^{\circ \alpha}\in \PP_n$ for every $A\in \PP_n$ and for every $\alpha \ge n-2$. For the case $\alpha<n-2$, Bhatia and Elsner~\cite{BE} studied an interesting class of $n \times n$ positive semi-definite Toeplitz matrices with real entries 
\begin{align} \label{def:C_n} C_n:= [\cos((i-j)\pi/n)]_{0\leq i,j\leq n-1}.\end{align}
They showed that for every even positive integer $n$, the matrix $|C_n|_{\circ}^{\circ\alpha}$ is not positive semi-definite if $n-4 < \alpha < n-2$. We note here that $S_{|A|_{\circ}}$ is also a monoid under addition and a closed subset of $[0, \infty)$, for any $A\in\PP_n$ for any $n$. Now from the result of~\cite{BE, hiai}, the monoid property and the fact that all the non-negative even integers belong to $S_{|A|_{\circ}}$ for every $A\in \PP_n$, it follows that $S_{|C_n|_{\circ}}=\{0, 2, \dots, n-4\}\cup[n-2,\infty)$ for even positive integers $n$. Jain~\cite{jain} studied these matrices for odd $n$ and proved that $S_{|C_n|_{\circ}}=\{0, 2, \dots, n-5\}\cup[n-3,\infty)$ for odd positive integers $n \ge 2$. Thus one has $\bigcap_{A\in \PP_n}S_{|A|_{\circ}}=\{0,2,\dots \}\cup[n-2,\infty)$ for all integers $n\ge 2$.

In this case also, one can ask about the possible form of the set $S_{|A|_{\circ}}$ for a fixed matrix $A\in\PP_n$. One can observe that if $A\in \PP_2$, then $S_{|A|_{\circ}}=[0,\,\infty)$. If $A\in \PP_3$, then it can be argued similarly as in the case of matrices with non-negative entries that $S_{|A|_{\circ}}=\{0\} \cup [\rho,\,\infty)$ for some $\rho\in[0,\,1]$. For $n\geq 4$, we give examples of $A\in \PP_n$ such that $S_{|A|_{\circ}}$ has more than one interval component. Let $n\ge 4$. Consider the matrix
 $$C_{n,\epsilon}=C
_n+\epsilon I_n, \quad \epsilon\geq 0,$$
where $C_n$ is the matrix defined in~\eqref{def:C_n} and $I_n$ is the identity matrix of order $n$. Note that $C_{n,\epsilon} \in \PP_n$ and so $[n-2, \infty) \subset S_{|C_{n,\epsilon}|_{\circ}}$ (follows from Theorem 5.1 of \cite{hiai}). We prove the following.
\begin{theorem}{\label{Thm cos}}
Fix integer $n\ge 4$. Let $\epsilon>0$ be sufficiently small (depending on $n$). Then there exists $\delta>0$ such that for any 
\begin{align*}
k\in \begin{cases}\{2,4,\dots,{n-4},n-2\}\,\,\, \text{  if  } \,n \,{  is \,\,even  },\\
\{2,4,\dots,{n-5},n-3\}\,\,\, \text{  if  } \,n \,{  is \,\,odd }, \end{cases}
\end{align*}
$[k-\delta, k+\delta]\subset S_{|C_{n,\epsilon}|_{\circ}}$. Also, for each such $k$ we have $k-1\notin S_{|C_{n,\epsilon}|_{\circ}}$. \end{theorem}

We prove the above theorem in Section~\ref{sec:proof cos}. 
\subsection*{A numerical example.} We give an example which shows that one can get at least two disjoint intervals when only the last diagonal entry of $A_n$ is perturbed. We give the following numerical example when $n=4$. Take $x_i=1/(i+1)$, $i=1,\dots,4$ and $\epsilon=10^{-7}$. Let $A_4$ be defined as in~\eqref{def:A_n} with these $x_i$s and let
\begin{align*}
A_{4,1,\epsilon}:=A_4+ \epsilon R_{4,1},
\end{align*}
where $R_{4,1}$ is the matrix of order $4$ with $1$ in the last diagonal entry and $0$ otherwise. Note that for $\alpha >1$, all the leading principal minors of order upto $3$ of $A_{4,1,\epsilon}^{\circ \alpha}$ are positive (follows from Theorem 2.6 of~\cite{jain}). Therefore the matrix $A_{4,1,\epsilon}^{\circ \alpha}$ is positive definite if and only if $\det(A_{4,1,\epsilon}^{\circ \alpha}) >0$. Now we plot $\det(A_{4,1,\epsilon}^{\circ \alpha})$ to see that there are two disjoint intervals where $A_{4,1,\epsilon}^{\circ \alpha}$ is positive definite.
\begin{figure}[h!]
\includegraphics[scale=.8]{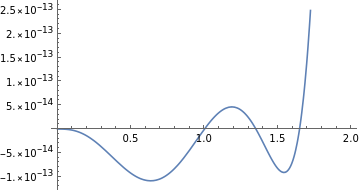}
\caption{$\det(A_{4,1,\epsilon}^{\circ \alpha})$ against $\alpha$.}
\end{figure}

We end this section with the following observation.
\begin{remark}
For two $n\times n$ matrices $A\geq B\geq 0$ with non-negative entries, define 
\begin{align*}
&S_{A,B}:=\{\alp\ge 0\suchthat A^{\circ \alpha}\geq B^{\circ \alpha}\} \\
&S'_{A,B}:=\{\alp\ge 0\suchthat (\lambda A+(1-\lambda) B)^{\circ \alpha}\leq \lambda A^{\circ \alpha}+ (1-\lambda)B^{\circ \alpha}\}.\end{align*}  Following Hiai (see Definition 1.1 of \cite{hiai}), $S_{A,B}$ (resp. $S'_{A,B}$) is the set of $\alp$ for which the function $x\mapsto x^\alp$ is Schur monotone (resp. Schur convex) for $A$ and $B$. From~\cite{FH77} and~\cite{hiai} we have
\begin{align*}
\bigcap_{A\geq B\geq 0} S_{A,B} =\{0, 1,\dots, n-2\} \cup [n-1, \infty)
\end{align*}
and
\begin{align*}
\bigcap_{A\geq B\geq 0} S^{'}_{A,B} =\{0, 1,\dots, n-1\} \cup [n, \infty).
\end{align*}
We remark that analogous results for these cases also hold. Indeed, Theorem \ref{thm1} shows that $S_{A_{n,k,\epsilon},0}$ and $S'_{A_{n,k,\epsilon},0}$ contain at least $k$ interval components each. Similar results hold for matrices with arbitrary real entries due to Theorem \ref{Thm cos}.
\end{remark}

\section{Proof of Theorem~\ref{thm1}}\label{sec:proofs}
In this section we prove Theorem~\ref{thm1}.
 \begin{proof}[Proof of Theorem~\ref{thm1}]
For any real symmetric matrix $M$, let $\lambda_{\min}(M)$ denote the smallest eigenvalue of $M$. We know that
\begin{align}\label{eq:min_eig}
\lambda_{\min}(M) = \min_{\|u\|=1} u^{T}Mu.
\end{align}
We first prove the existence of $\alpha_\ell$ for $\ell\in\{n-k-1, n-k,\dots,n-3\}$. From Theorem 2.6 of~\cite{jain} we have  
\begin{align}\label{eq:lam_min}
\lambda_{\min}(A_n^{\circ\alpha})\begin{cases}
    = 0 &   \alpha\in \{0,1,\dots,n-2\},\\
    < 0 &  0<\alpha<n-2,\alpha\notin \{1,2,\dots,n-2\}.
    \end{cases}
\end{align}
Note that $\lambda_{\min}(A_{n,k,\epsilon}^{\circ \alpha})$ is a continuous function in $\epsilon$ and $\alpha$. So, using~\eqref{eq:lam_min} we get the existence of the points $\alpha_\ell$ for small enough $\epsilon$.

 We fix such an $\epsilon$. Since $S_{A_{n,k,\epsilon}}$ is a monoid containing $1$, it is enough to prove the existence of $\delta$ such that $[n-k-1-\delta,\, n-k-1+\delta] \subset S_{A_{n,k,\epsilon}}$. Let $v^{(\alpha)} = (v_1^{(\alpha)},\dots, v_n^{(\alpha)})^T$ be a unit eigenvector corresponding to the eigenvalue $\lambda_{\min}(A_{n,k,\epsilon}^{\circ \alpha})$ of $A_{n,k,\epsilon}^{\circ \alpha}$. Now we consider the following two cases.

\textbf{Case 1:} Suppose all of $v_{n-k+1}^{(\alpha)}, v_{n-k+2}^{(\alpha)}, \dots ,v_{n}^{(\alpha)}$ go to $0$ as $\alpha$ tends to $n-k-1$.\\
Define $u^{(\alpha)}:= (v_1^{(\alpha)},\dots, v_{n-k}^{(\alpha)})$. Then we have
\begin{align}\label{eq:last}
\lambda_{\min}(A_{n,k,\epsilon}^{\circ \alpha}) &= (v^{(\alpha)})^T A_{n,k,\epsilon}^{\circ \alpha} v^{(\alpha)} = (u^{(\alpha)})^T A_{n-k}^{\circ \alpha} u^{(\alpha)} + h(\alpha)\nonumber\\
& \ge \lambda_{\min}(A_{n-k}^{\circ \alpha}) \|u^{(\alpha)}\|^2+ h(\alpha),
\end{align}
where $h(\alpha)$ is the sum of terms in $(v^{(\alpha)})^T A_{n,k,\epsilon}^{\circ \alpha} v^{(\alpha)}$ in which at least one of $v_{n-k+1}^{(\alpha)}, v_{n-k+2}^{(\alpha)}, \dots ,v_{n}^{(\alpha)}$ is present. The inequality in~\eqref{eq:last} follows from~\eqref{eq:min_eig}. Note that $h(\alpha) \rightarrow 0$ as $\alpha \rightarrow n-k-1$. Also, $\|u^{(\alpha)}\| \rightarrow 1$ as $\alpha \rightarrow n-k-1$.
Again, $A_{n-k}^{\circ(n-k-1)}$ is p.s.d. and from Theorem 2.6 of~\cite{jain} we have that $\det(A_{n-k}^{\circ(n-k-1)})>0$. So, $\lambda_{\min}(A_{n-k}^{\circ(n-k-1)}) >0$. Thus it follows from~\eqref{eq:last} that there exists $\delta>0$ such that $\lambda_{\min}(A_{n,k,\epsilon}^{\circ \alpha}) >0$ for all $\alpha\in[n-k-1-\delta,\, n-k-1+\delta]$.

\textbf{Case 2:} Suppose at least one of $v_{n-k+1}^{(\alpha)}, v_{n-k+2}^{(\alpha)}, \dots ,v_{n}^{(\alpha)}$does not go to $0$ as $\alpha$ tends to $n-k-1$. Without loss of generality, suppose there exists $\eta_1>0$ and a sequence $\alpha_m$ converging to $n-k-1$ such that $|v_{n}^{(\alpha_m)}| > \eta_1$ for all $m$.\\
We have
\begin{align*}
\lambda_{\min}(A_{n,k,\epsilon}^{\circ \alpha}) &= (v^{(\alpha)})^T A_{n,k,\epsilon}^{\circ \alpha} v^{(\alpha)}\\
& = (v^{(\alpha)})^T A_{n}^{\circ \alpha} v^{(\alpha)} +\sum_{i=1}^{k-1} \left( (1+x_{n-k+i}^2+\epsilon)^\alpha -(1+x_{n-k+i}^2)^\alpha \right)(v_{n-k+i}^{(\alpha)})^2 \\
&\qquad+ \left( (1+x_{n}^2+\epsilon)^\alpha -(1+x_{n}^2)^\alpha \right)(v_{n}^{(\alpha)})^2.
\end{align*}
Note that there exists $\eta_2>0$ such that $(1+x_{n}^2+\epsilon)^\alpha -(1+x_{n}^2)^\alpha > \eta_2$ for all $\alpha \in [n-k-2, n-k]$.
Then we have for all $m$ such that $\alpha_m\in[n-k-2, n-k]$
\begin{align*}
\lambda_{\min}(A_{n,k,\epsilon}^{\circ \alpha_m}) \ge (v^{(\alpha_m)})^T A_{n}^{\circ \alpha_m} v^{(\alpha_m)} + \eta_1^2\eta_2 \ge \lambda_{\min}(A_{n}^{\circ \alpha_m}) + \eta_1^2\eta_2. 
\end{align*}
The last inequality follows from~\eqref{eq:min_eig}. But $A_{n}^{\circ(n-k-1)}$ is p.s.d. and from Theorem 2.6 of~\cite{jain} we have that $\det(A_{n}^{\circ(n-k-1)})=0$. Hence $\lambda_{\min}(A_{n}^{\circ(n-k-1)})=0$. Therefore, letting $m$ go to infinity and using the continuity of the smallest eigenvalue we get that $\lambda_{\min}(A_{n,k,\epsilon}^{\circ(n-k-1)}) >0$.
This implies that there exists $\delta>0$ such that $\lambda_{\min}(A_{n,k,\epsilon}^{\circ \alpha}) >0$ for all $\alpha\in [n-k-1-\delta,\, n-k-1+\delta]$.
This completes the proof of the first part of the theorem.

Finally, note that from Theorem 2.6 of~\cite{jain} we have $\lambda_{\min}(A_{n-k}^{\circ \alpha}) <0$ for all non-integer $\alpha\in[0,\,n-k-2]$. Hence by Cauchy's interlacing theorem ( Theorem 4.3.17 of~\cite{HJ} ) we have $\lambda_{\min}(A_{n,k,\epsilon}^{\circ \alpha}) <0$ for all non-integer $\alpha\in[0,\,n-k-2]$. This proves the last part of the theorem.
 \end{proof}

\section{Proof of Theorem~\ref{Thm cos}}\label{sec:proof cos}
We now prove Theorem~\ref{Thm cos}.
\begin{proof}[Proof of Theorem~\ref{Thm cos}]
First suppose $n$ is an even integer. Since $S_{|C_n|_{\circ}}$ is a monoid containing $2$, Theorem $2$ of~\cite{BE} gives $\lambda_{\min}(|C_n|_{\circ}^{\circ\alpha}) < 0$ for   $0<\alpha<n-2$ with $\alpha\notin \{2,4,\dots,n-2\}$. Also, by the Schur product theorem we have $\lambda_{\min}(|C_n|_{\circ}^{\circ\alpha}) \ge 0$ for $\alpha\in\{2,4,\dots,n-2\}$. Therefore, using the continuity of $\lambda_{\min}(|C_n|_{\circ}^{\circ\alpha})$, we conclude that $\lambda_{\min}(|C_n|_{\circ}^{\circ\alpha}) = 0$ for $\alpha\in\{2,4,\dots,n-2\}$. Thus we have
\begin{align*}
\lambda_{\min}(|C_n|_{\circ}^{\circ\alpha})\begin{cases}
    = 0 &   \alpha\in \{0,2,4,\dots,n-2\},\\
    < 0 &  0<\alpha<n-2,\alpha\notin \{2,4,\dots,n-2\}.
    \end{cases}
\end{align*}

Let $$\lambda_0:= \min\{|\lambda_{\min}(|C_n|_{\circ}^{\circ(k+1)})|: 0\le k \le n-4, k\,\text{even}\, \}.$$
Define $$H_{n,\epsilon,\alpha}:=|C_{n, \epsilon}|_{\circ}^{\circ\alpha}-|C_n|_{\circ}^{\circ\alpha}= ((1+\epsilon)^\alpha -1) I_n$$ for $0\le \alpha \le n-2$.
We choose $\epsilon>0$ small enough so that $$\gamma:= \sup\{(1+\epsilon)^\alpha -1: 0\le \alpha \le n-2\} < \lambda_0.$$ For such choice of $\epsilon$, let $$\gamma_0:= \inf\{(1+\epsilon)^\alpha -1: r\le \alpha \le n-2\},$$
where $0<r<1$ fixed. Clearly, $\gamma_0 >0$. Now using the continuity of $\lambda_{\min}(|C_n|_{\circ}^{\circ\alpha})$ and the fact that $\lambda_{\min}(|C_n|_{\circ}^{\circ\alpha})=0$ for $\alpha\in \{2,4,\dots,n-2\}$, we find $\delta>0$ such that for $\alpha\in [k-\delta, k+\delta]$, $|\lambda_{\min}(|C_n|_{\circ}^{\circ\alpha})|\leq \gamma_0$ for $k\in\{2,4,\dots,n-4\}.$

We now use Weyl's inequality (Theorem 4.3.1 \cite{HJ}) to get bounds on the smallest eigenvalue of $|C_{n, \epsilon}|_{\circ}^{\circ\alpha}$. Note that for $0\le \alpha \le n-2$, $|C_{n, \epsilon}|_{\circ}^{\circ\alpha}= |C_n|_{\circ}^{\circ\alpha} + H_{n,\epsilon,\alpha}$. Therefore by Theorem 4.3.1 of~\cite{HJ} we have 
\begin{align*}
\lambda_{\min}(|C_n|_{\circ}^{\circ\alpha}) + \gamma_0 \le \lambda_{\min}(|C_{n,\epsilon}|_{\circ}^{\circ\alpha})\le \lambda_{\min}(|C_n|_{\circ}^{\circ\alpha}) + \gamma
\end{align*}
for $r\le \alpha \le n-2$.

Thus we proved that for $k\in\{2,4,\dots,n-4, n-2\}$, $\lambda_{\min}(|C_{n,\epsilon}|_{\circ}^{\circ\alpha})>0$ for $\alpha\in [k-\delta, k+\delta]$ and $\lambda_{\min}(|C_{n, \epsilon}|_{\circ}^{\circ(k-1)})<0$. This concludes the proof when $n$ is even.

Now suppose $n$ is odd. Then using Theorem $3.3$ of~\cite{jain} we get
\begin{align*}
\lambda_{\min}(|C_n|_{\circ}^{\circ\alpha})\begin{cases}
    = 0 &   \alpha\in \{0,2,4,\dots,n-3\},\\
    < 0 &  0<\alpha<n-3,\alpha\notin \{2,4,\dots,n-5\},\\
    >0  &  \alpha>n-3.
    \end{cases}
\end{align*}
The rest of the proof for this case is similar to the case when $n$ is even and is omitted.
\end{proof}

\subsection*{Acknowledgement.} The authors thank Manjunath Krishnapur for suggesting the question addressed in this article and for helpful discussions. They thank Apoorva Khare for his insightful suggestions on a draft of this article, and an anonymous referee for insightful remarks.

\bibliographystyle{abbrvnat}
\bibliography{biblio_lin_alg}

\begin{thebibliography}{9}
\providecommand{\natexlab}[1]{#1}
\providecommand{\url}[1]{\texttt{#1}}
\expandafter\ifx\csname urlstyle\endcsname\relax
  \providecommand{\doi}[1]{doi: #1}\else
  \providecommand{\doi}{doi: \begingroup \urlstyle{rm}\Url}\fi

\bibitem[Bhatia and Elsner(2007)]{BE}
R.~Bhatia and L.~Elsner.
\newblock Positivity preserving {H}adamard matrix functions.
\newblock \emph{Positivity}, 11\penalty0 (4):\penalty0 583--588, 2007.

\bibitem[FitzGerald and Horn(1977)]{FH77}
C.~H. FitzGerald and R.~A. Horn.
\newblock On fractional {H}adamard powers of positive definite matrices.
\newblock \emph{Journal of Mathematical Analysis and Applications}, 61\penalty0
  (3):\penalty0 633--642, 1977.

\bibitem[Guillot et~al.(2015)Guillot, Khare, and Rajaratnam]{GKB}
D.~Guillot, A.~Khare, and B.~Rajaratnam.
\newblock Complete characterization of {H}adamard powers preserving {L}oewner
  positivity, monotonicity, and convexity.
\newblock \emph{Journal of Mathematical Analysis and Applications},
  425\penalty0 (1):\penalty0 489--507, 2015.

\bibitem[Hiai(2009)]{hiai}
F.~Hiai.
\newblock Monotonicity for entrywise functions of matrices.
\newblock \emph{Linear Algebra and its Applications}, 431\penalty0
  (8):\penalty0 1125--1146, 2009.

\bibitem[Horn and Johnson(2013)]{HJ}
R.~A. Horn and C.~R. Johnson.
\newblock \emph{Matrix analysis}.
\newblock Second ed. Cambridge University Press, 2013.

\bibitem[Jain(2017)]{jain}
T.~Jain.
\newblock Hadamard powers of some positive matrices.
\newblock \emph{Linear Algebra and its Applications}, 528:\penalty0 147--158,
  2017.

\bibitem[Khare(2020)]{khare2020}
A.~Khare.
\newblock Multiply positive functions, critical exponent phenomena, and the
  {J}ain-{K}arlin-{S}choenberg kernel.
\newblock \emph{arXiv preprint arXiv:2008.05121}, 2020.

\bibitem[Rudin(1959)]{WR}
W.~Rudin.
\newblock Positive definite sequences and absolutely monotonic functions.
\newblock \emph{Duke Mathematical Journal}, 26\penalty0 (4):\penalty0 617--622,
  1959.

\bibitem[Schoenberg(1942)]{IS}
I.~J. Schoenberg.
\newblock Positive definite functions on spheres.
\newblock \emph{Duke Mathematical Journal}, 9\penalty0 (1):\penalty0 96--108,
  1942.

\end{thebibliography}

\end{document}